\documentclass[12pt]{amsart}
\usepackage{colonequals}
\usepackage[alphabetic]{amsrefs}
\usepackage{enumerate}
\usepackage{amssymb}
\usepackage{fullpage}

\newtheorem{theorem}{Theorem}[section]
\newtheorem{lemma}[theorem]{Lemma}
\newtheorem{proposition}[theorem]{Proposition}
\newtheorem{corollary}[theorem]{Corollary}
\usepackage[colorlinks=true, pdfstartview=FitH, linkcolor=blue, citecolor=blue, urlcolor=blue]{hyperref}

\theoremstyle{definition}

\newtheorem{example}[theorem]{Example}
\newtheorem{question}[theorem]{Question}

\theoremstyle{remark}
\newtheorem{remark}[theorem]{Remark}

\numberwithin{equation}{section}

\newcommand{\bP}{\mathbb{P}}
\renewcommand{\char}{{\rm char}}
\newcommand{\Fp}{\mathbb{F}_p}
\newcommand{\Fq}{\mathbb{F}_q}

\makeatletter
\@namedef{subjclassname@2020}{%
  \textup{2020} Mathematics Subject Classification}
  \makeatother

\begin{document}

\title{Tangent-filling plane curves over finite fields}

\author{Shamil Asgarli}
\address{Department of Mathematics and Computer Science \\ Santa Clara University \\ 500 El Camino Real \\ USA 95053}
\email{sasgarli@scu.edu}

\author{Dragos Ghioca}
\address{Department of Mathematics, University of British Columbia, Vancouver, BC V6T 1Z2}
\email{dghioca@math.ubc.ca}


\begin{abstract}
We study plane curves over finite fields whose tangent lines at smooth $\mathbb{F}_q$-points together cover all the points of $\mathbb{P}^2(\mathbb{F}_q)$.
\end{abstract}

\maketitle

\section{Introduction}

The investigation of algebraic curves over finite fields is an ever-growing research topic. Stemming from the intersection of algebra, number theory and algebraic geometry, it influences a wide array of fields such as coding theory and combinatorial design theory \cite{HKT08}. As one specific example in this vast body of work, finding curves with many $\mathbb{F}_q$-rational points remains an interesting challenge. The motivation behind searching for extremal curves ranges from purely theoretical reasons (e.g. understanding the sharpness of Hasse-Weil inequality) to more applied constructions (e.g. obtaining a rich configuration of points). 

It is already instructive to restrict attention to plane curves. We list a few different definitions from the literature for a given projective irreducible plane curve $C\subset\mathbb{P}^2$ of degree $d$ over a finite field $\mathbb{F}_q$ to have ``a lot of $\Fq$-rational points".
\begin{enumerate}[(a)]
\item \label{def:maximal-curve} We say that $C$ is a \emph{maximal curve} if $\#C(\Fq) = q+1 + (d-1)(d-2)\sqrt{q}$, namely, the curve achieves the equality in the Hasse-Weil upper bound for its $\Fq$-rational points.
\item \label{def:plane-filling-curve} We say that $C$ is \emph{plane-filling} if $C(\Fq)=\bP^2(\Fq)$, that is, $C$ contains each of the $q^2+q+1$ distinct $\Fq$-points of $\mathbb{P}^2$. 
\item \label{def:blocking-curve}  We say that $C$ is \emph{blocking} if $C(\Fq)$ is a blocking set, that is, $C$ meets every $\Fq$-line $L$ at some $\Fq$-point.
\end{enumerate}

The main purpose of the present paper is to introduce a new concept that indicates in yet another way that the curve contains many $\Fq$-points. 
\begin{enumerate}[(a)]
 \setcounter{enumi}{3}
    \item \label{def:tangent-filling-curve} We say that $C$ is \emph{tangent-filling} if every point $P\in \mathbb{P}^2(\Fq)$ lies on a tangent line $T_Q C$ to the curve $C$ at some smooth $\Fq$-point $Q$.
\end{enumerate} 

Regarding the literature, we note that curves satisfying \eqref{def:maximal-curve} have been thoroughly studied in many papers ranging from foundational work \cites{CHKT00, GK09, GGS10} to the more recent discoveries \cites{BM18, BLM23}. The curves satisfying \eqref{def:plane-filling-curve} have been analyzed by \cites{HK13, Dur18, Hom20}. Finally, the curves satisfying \eqref{def:blocking-curve} have been recently examined by the authors in joint work with Yip \cites{AGY22a, AGY22b, AGY23a, AGY23b}.

Our first theorem shows that a curve of a low degree cannot be tangent-filling. We first state the result when the ground field is $\mathbb{F}_p$ for some prime $p$. For convenience, we state the result for $d\geq 3$ and discuss the case $d=2$ in Remark~\ref{rem:conic}.

\begin{theorem}\label{thm:low-degree-p}
Let $C\subset\mathbb{P}^2$ be an irreducible plane curve of degree $d\geq 3$ defined over $\mathbb{F}_p$ where $p$ is a prime. If $p\geq 4(d-1)^2(d-2)^2$, then $C$ is not tangent-filling.
\end{theorem}

We have an analogous result for an arbitrary finite field $\mathbb{F}_q$. 

\begin{theorem}\label{thm:low-degree-q}
Let $C\subset\mathbb{P}^2$ be an irreducible plane curve of degree $d\geq 2$ defined over $\mathbb{F}_q$. If $p>d$ and $q\geq d^2(d-1)^6$, then $C$ is not tangent-filling.
\end{theorem}

Let us briefly compare the bounds in these two theorems. The bound $p\geq O(d^4)$ in Theorem~\ref{thm:low-degree-p} is replaced with a pair of bounds $p>d$ and $q\geq O(d^8)$ in Theorem~\ref{thm:low-degree-q}. From one perspective, Theorem~\ref{thm:low-degree-q} provides worse bounds on $q$, and it remains open to improve $q\geq O(d^8)$ to $q\geq O(d^4)$. From another perspective, Theorem~\ref{thm:low-degree-q} provides better bounds on the characteristic $p$; for instance, when $q=p^n$ with $n=4$, the bound $p^4=q\geq O(d^8)$ is equivalent to $p\geq O(d^2)$, which is a weaker hypothesis than the earlier bound $p\geq O(d^4)$. It is also natural to consider the situation where we restrict our attention to a more restrictive class of smooth curves; in this case, Remark~\ref{rem:kaji} explains to obtain a slightly improved result.

We are also interested in finding examples of tangent-filling curves. Clearly, any smooth plane-filling curve is tangent-filling. Since the degree of the smallest plane-filling curve over $\mathbb{F}_q$ is $q+2$ by \cite{HK13}, it is natural to search for tangent-filling curves with degrees less than $q+2$. Our next theorem exhibits an example of a tangent-filling curve of degree $q-1$.

\begin{theorem}\label{thm:special-curve}
    Let $q\geq 11$ and $p=\operatorname{char}(\mathbb{F}_q)>3$. The curve $C$  defined by the equation
    $$
    x^{q-1} + y^{q-1} + z^{q-1} - 3(x+y+z)^{q-1} = 0
    $$
    is an irreducible tangent-filling curve. 
\end{theorem}

\begin{remark}
\label{rem:alt_char}
We note that if $\char(\Fq)=2$ in Theorem~\ref{thm:special-curve}, then the curve $C$ is reducible, as it contains the lines $x=y$, $y=z$ and $z=x$. 

On the other hand, if $\char(\Fq)=3$, then curve $C$ in Theorem~\ref{thm:special-curve} is smooth, but it is not tangent-filling since no tangent line at a point of $C(\Fq)$ passes through any of the points $[1:0:0]$, $[0:1:0]$ and $[0:0:1]$. This claim can be easily checked since the points $[x_0:y_0:z_0]\in C(\Fq)$ have the property that $x_0y_0z_0\ne 0$ (the proof of this fact follows similarly as in Lemma~\ref{lem:rat-points}, which characterizes the $\Fq$-points of $C$ when $\char(\Fq)>3$), while the equation of the tangent line at the point $[x_0:y_0:z_0]\in C(\Fq)$ is
$$x_0^{q-2}\cdot x + y_0^{q-2}\cdot y + z_0^{q-2}\cdot z = 0.$$ 
Finally, a simple computer check shows that the curve $C$ from Theorem~\ref{thm:special-curve} is not tangent-filling when $q\in\{5,7\}$ (see also Remark~\ref{rem:q_7}).
\end{remark}

While we expect that $d=q-1$ is not the smallest possible degree of a tangent-filling curve, we believe that Theorem~\ref{thm:special-curve} is novel in several ways. First, checking the tangent-filling condition over $\mathbb{F}_q$ requires careful analysis of the $\mathbb{F}_q$-points of the curve. Second,   in our previous work with Yip \cites{AGY22a}, we found several families of blocking smooth curves of degree less than $q$ and so, it was natural to test those families whether they are also tangent-filling; however, none of the tested families of blocking smooth curves turned out to be tangent-filling. This suggests that finding tangent-filling curves may be very challenging, much more than the case of blocking curves. In particular, finding tangent-filling curves of degree less than $q$ seems very difficult in general.  Quite interestingly, the curve from Theorem~\ref{thm:special-curve} is \emph{not} blocking since $C(\Fq)$ does not intersect the $\Fq$-lines $x=0$, $y=0$, $z=0$ and $x+y+z=0$ (see also Corollary~\ref{cor:non-blocking}).

We remark that when $q$ has a special form, there are more optimal examples. The noteworthy example is the Hermitian curve $\mathcal{H}_q$ defined by $x^{\sqrt{q}+1} + y^{\sqrt{q}+1} + z^{\sqrt{q}+1} = 0$ when $q$ is a square. We will see in Example~\ref{example:hermitian} that $\mathcal{H}_q$ is a tangent-filling curve. Thus, for $q$ square, there is a (smooth) tangent-filling curve of degree $\sqrt{q}+1$. 

Inspired by the example in the previous paragraph, we may ask for the most optimal curve that has the tangent-filling property. 
\begin{question}\label{quest:optimal-tangent-filling-curve}
What is the minimum degree of an irreducible tangent-filling plane curve over $\mathbb{F}_q$?
\end{question} 

Let us explain a heuristic that suggests that the optimal degree may not be too far away from $\sqrt{q}$ even for a general $q$. Consider a collection $\mathcal{L}$ of $\Fq$-lines such that 
\begin{equation}\label{eq:blocking-set-in-dual-space}
\bigcup_{L\in \mathcal{L}} L(\Fq) = \mathbb{P}^2(\Fq).
\end{equation}
By viewing each line as a point in the dual space $(\mathbb{P}^2)^{\ast}$, the condition~\eqref{eq:blocking-set-in-dual-space} is equivalent to $\mathcal{L}$ being a blocking set in $(\mathbb{P}^2)^{\ast}(\Fq)$. There are plenty of blocking sets with size a constant multiple of $q$; for instance, the so-called \emph{projective triangle}, a well-known example of a blocking set, has $\frac{3}{2}(q+1)$ points for odd $q$ \cite{Hir79}. So, we choose $\mathcal{L}$ that satisfies \eqref{eq:blocking-set-in-dual-space} and $|\mathcal{L}|=c_0 q$ for some constant $c_0>0$. Next, suppose that it is possible to pick distinct $\Fq$-points $P_i\in L_i$ for each $L_i\in\mathcal{L}$, so that $P_i\neq P_j$ for $i\neq j$. Let us impose the condition that a degree $d$ curve passes through the point $P_i$ and has contact order at least $2$ with the line $L_i$ at the point $P_i$. For each value of $i$, this imposes $2$ linear conditions in the parameter space $\mathbb{P}^N$ of plane curves of degree $d$, where $N=\binom{d+2}{2}-1$. Assuming that $\binom{d+2}{2}-1 > 2 |\mathcal{L}| = 2c_0q$, we obtain a curve of degree $d$ satisfying each of these local conditions. By construction, each such curve is tangent to the line $L_i$ at the point $P_i$, and tangent-filling property is enforced by \eqref{eq:blocking-set-in-dual-space}. The main issue is that all such resulting curves may be singular at one (or more) of the points $P_i$. While the bound of the form $d > c_1\sqrt{q}$ for some constant $c_1>0$ is predicted by this heuristic, it seems very challenging to make this interpolation argument precise. 

\subsection*{Structure of the paper.} In Section~\ref{sect:low-degree-curves}, we borrow tools from classical algebraic geometry and combinatorics of blocking sets to prove our Theorems~\ref{thm:low-degree-q}~and~\ref{thm:low-degree-p}. In Section~\ref{sect:family-curves}, we prove Theorem~\ref{thm:special-curve} by studying in detail the geometric properties of the given curve $C$, such as its singular locus and irreducibility, along with an arithmetic analysis for the equation of a tangent line at a smooth $\Fq$-point of $C$.

\subsection*{Acknowledgments.} We thank the anonymous referee for their useful comments and suggestions, which improved our presentation.

\section{Curves of low degree are not tangent-filling} \label{sect:low-degree-curves}

In this section, we prove Theorem~\ref{thm:low-degree-q} and Theorem~\ref{thm:low-degree-p}. We start with preliminary geometric constructions. Given a plane curve $C$, recall that the dual curve $C^{\ast}$ parametrizes the tangent lines to $C$. More formally, $C^{\ast}$ is the closure of the image of the Gauss map $\gamma_G\colon C\to (\mathbb{P}^{2})^{\ast}$ mapping a regular point $P$ on $C$ to the line $T_P C$. 

When the Gauss map $\gamma_G$ is separable, the geometry of the tangent lines to the curve in characteristic $p$ is similar to the behaviour observed in characteristic $0$. It turns out that the curve $C$ is \emph{reflexive} (that is, the double dual $(C^{\ast})^{\ast}$ can be canonically identified with $C$ itself) if and only if $\gamma_G$ is separable \cite{Wal56}. Thus, all curves in characteristic $0$ are reflexive. In positive characteristic $p$, the condition $p>d$ is sufficient to ensure that a plane curve of degree $d$ is reflexive \cite{Par86}*{Proposition 1.5}.

\subsection{Bitangents} For a given plane curve $C$, we say that a line $L$ is \emph{bitangent} to $C$ if $L$ is tangent to the curve $C$ in at least two points. The following is a well-known result in classical algebraic geometry; we include its proof to emphasize how the hypothesis $p>d$ is used. Since it is possible to have a curve with infinitely many bitangents \cite{Pie94}*{Example 2}, the lemma below would not be true if we completely remove the assumption $p>d$. 

\begin{lemma}\label{lem:bitangents}
Let $C\subset\mathbb{P}^2$ be a geometrically irreducible plane curve of degree $d\geq 2$ defined over $\mathbb{F}_q$ such that $p>d$. Then $C$ has at most $ \frac{1}{2}d^2(d-1)^2$ many bitangents.
\end{lemma} 

\begin{proof}
The condition $p>d$ guarantees that the Gauss map $\gamma_G$ is separable. The dual curve $C^{\ast}$ has degree $\delta \leq d(d-1)$. Since $C^{\ast}$ is geometrically irreducible, it has at most $\binom{\delta-1}{2}$ many singular points \cite{Har92}*{Exercise 20.18}. Every bitangent of the curve $C$ corresponds to some singular point of $C^{\ast}$, because $\gamma_G$ is separable \cite{Wal56}. Thus, the number of bitangents to $C$ is at most
$$
\binom{\delta-1}{2} \leq \binom{d(d-1)-1}{2}
 = \frac{1}{2}\left(d^2-d-1\right)\left(d^2-d-2\right) \leq \frac{1}{2}\left(d^2-d\right)\left(d^2-d\right)
$$
as desired. \end{proof} 

The previous lemma would hold if we replaced the hypothesis $p>d$ with the weaker hypothesis that the Gauss map of $C$ is separable. 

\subsection{Strange curves} We say that an irreducible plane curve $C$ of degree $d\geq 2$ over a field $K$ is \emph{strange} if all the tangent lines to the curve $C$ at its smooth $\overline{K}$-points are concurrent. This is equivalent to requiring that the dual curve of $C$ is a line. Since the double dual of a strange curve cannot be the original curve, it follows that strange curves must be nonreflexive. In particular, strange curves can only exist when $p=\operatorname{char}(K)>0$. Strange curves do exist \cite{Pie94}*{Example 1}: for instance, all the tangent lines to the curve $xy^{p-1}-z^p=0$ pass through the point $[0:0:1]$. The paper \cite{BH91} contains several results on various properties and characterizations of strange curves.

As mentioned in the beginning of the section, the hypothesis $p>d$ ensures that the curve is reflexive. Thus, a plane curve of degree $d\geq 2$ cannot be strange whenever $p>d$. This fact will be crucially used in the proofs below, when we verify that the $\Fq$-points of the dual curve $C^{\ast}$ do not produce a trivial blocking set.

\subsection{Proofs of Theorem~\ref{thm:low-degree-p} and Theorem~\ref{thm:low-degree-q}.} We now present the proof of our first main theorem which roughly states that tangent-filling curves over $\Fp$ cannot exist when $p$ is larger than a quadratic function of $d$.

\begin{proof}[Proof of Theorem~\ref{thm:low-degree-p}] We first assume that $C$ is geometrically irreducible.  We start by observing that the hypothesis $p\geq 4(d-1)^2(d-2)^2$ implies $p>d$ for $d\geq 3$. Thus, the curve $C$ is reflexive, and in particular, $C$ is not strange, meaning that $\deg(C^{\ast})>1$. By applying Hasse-Weil bound \cite{AP96}*{Corollary 2.5}, we have
\begin{align*}
    \# C(\Fp) \leq p+1 + (d-1)(d-2) \sqrt{p}.
\end{align*}
Suppose, to the contrary, that $C$ is tangent-filling. Let $B\subseteq C^{\ast}(\Fq)$ correspond to the set of tangent $\Fp$-lines to the curve $C$ at smooth $\Fp$-points. It is clear that
\begin{align}\label{ineq:upper-bound-B}
      \# B \leq \# C(\Fp)\leq p+1 + (d-1)(d-2) \sqrt{p}.
\end{align}
Note that $B$ is a blocking set by definition of tangent-filling; indeed, each $\Fp$-line in the dual projective plane parametrizes lines passing through a fixed point, so $B$ meets every $\Fp$-line in the dual space. Since $1<\deg(C^{\ast})\leq d(d-1) < p+1$, the set $B$ is a non-trivial blocking set, that is, $B$ cannot contain all the $\Fp$-points of some $\Fp$-line in $(\mathbb{P}^2)^{\ast}(\Fp)$. Indeed, $C^{\ast}$ is irreducible (as it is the image of the irreducible curve $C$ through the map $\gamma_G$) and has degree less than $p+1$. By Blokhuis' theorem ~\cite{Blo94}, 
\begin{align}\label{ineq:lower-bound-B}
\#B\geq \frac{3}{2}(p+1).
\end{align}
Combining \eqref{ineq:upper-bound-B} and \eqref{ineq:lower-bound-B}, we get $p+1 + (d-1)(d-2) \sqrt{p} \geq \frac{3}{2}(p+1)$ which contradicts the hypothesis $p\geq 4(d-1)^2(d-2)^2$.

Now, suppose that $C$ is not geometrically irreducible. Since $C$ is irreducible but not geometrically irreducible, we conclude that $\# C(\Fp) \leq \frac{d^2}{4}$ (see \cite{CM06}*{Lemma 2.3} or \cite{AG23}*{Remark 2.2}). In particular, the number of distinct tangent $\mathbb{F}_p$-tangent lines to $C$ is at most $\frac{d^2}{4}$. Since each $\Fp$-line covers $p+1$ points of $\mathbb{P}^2(\Fp)$, all the tangent lines to $C$ at its smooth $\Fp$-points together can cover at most $\frac{d^2}{4}\cdot (p+1)$
distinct $\Fq$-points. Since $p\geq 4(d-1)^2(d-2)^2$, it is immediate that $\frac{d^2}{4}\cdot (p+1)<p^2+p+1$, so $C$ is not tangent-filling.  \end{proof}

\begin{remark}\label{rem:conic}
Note that the inequality $p\geq 4(d-1)^2(d-2)^2$ automatically implies $p>d$ when $d\geq 3$. However, when $d=2$, the inequality  $p\geq 4(d-1)^2(d-2)^2$ is vacuous, and $p=2$ is allowed. When $p=2$ and $d=2$, the smooth conics are strange curves, which are therefore tangent-filling because the tangent lines at the $\Fq$-rational points of this conic are all the $q+1$ lines passing through some given point in $\bP^2(\Fq)$. So, Theorem~\ref{thm:low-degree-p} does not hold when $p=d=2$; on the other hand, Theorem~\ref{thm:low-degree-p} continues to hold when $d=2$ and $p>2$ with essentially the same proof as the one above.
\end{remark}

We proceed to prove our second main result concerning tangent-filling curves over an arbitrary finite field $\mathbb{F}_q$. 

\begin{proof}[Proof of Theorem~\ref{thm:low-degree-q}] We first assume that the curve $C$ is geometrically irreducible, that is, irreducible over $\overline{\mathbb{F}_q}$. We claim that $C^{\ast}$ is not a blocking curve. Suppose, to the contrary, that $C^{\ast}(\Fq)$ is a blocking set in $(\mathbb{P}^2)^{\ast}(\Fq)$. Since $p>d$, the curve $C$ is not strange, that is, $\deg(C^{\ast})>1$. Since $1<\deg(C^{\ast})\leq d(d-1) < q+1$, the set $B$ is a \emph{non-trivial} blocking set by the same reasoning given in the proof of Theorem~\ref{thm:low-degree-p}. By \cite{AGY23a}*{Lemma 4.1}, 
\begin{equation}\label{ineq:lower-bound-dual}
    \# C^{\ast}(\Fq) > q + \frac{q+\sqrt{q}}{\deg(C^{\ast})} \geq q + \frac{q+\sqrt{q}}{d(d-1)}.
\end{equation}
On the other hand, the number of $\Fq$-points on the dual curve $C^{\ast}$ is bounded above:
\begin{equation}\label{ineq:upper-bound-dual}
    \# C^{\ast}(\Fq)\leq \#C(\Fq) + \# \{\text{bitangents to } C \text{ defined over } \Fq\}.
\end{equation}
Combining Lemma~\ref{lem:bitangents}, Hasse-Weil bound applied to $C$ \cite{AP96}*{Corollary 2.5}, and \eqref{ineq:upper-bound-dual}, we obtain an upper bound:
\begin{equation}\label{ineq:upper-bound-dual-explicit}
    \# C^{\ast}(\Fq)\leq q+1 + (d-1)(d-2)\sqrt{q} + \frac{1}{2} d^2(d-1)^2.
\end{equation}
Comparing \eqref{ineq:lower-bound-dual} and \eqref{ineq:upper-bound-dual-explicit}, we obtain
$$
(d-1)(d-2)\sqrt{q} + \frac{1}{2} d^2(d-1)^2 + 1 > \frac{q+\sqrt{q}}{d(d-1)},
$$
or equivalently,
\begin{equation}\label{ineq:contradictory}
    d(d-1)^2(d-2)\sqrt{q} + \frac{1}{2}d^3(d-1)^3 + d(d-1) > q + \sqrt{q}
\end{equation}
Since $\sqrt{q}\geq d(d-1)^3$, we have $\sqrt{q}\geq \frac{1}{2} d^2(d-1)$ which allows us to deduce,
\begin{align*}
    q + \sqrt{q} &\geq d(d-1)^2 \cdot \left((d-1)\sqrt{q}\right) + \sqrt{q} \\
    &\geq d(d-1)^2 \cdot \left((d-2)\sqrt{q}+\sqrt{q}\right) + d(d-1) \\
    &\geq d(d-1)^2 \cdot \left((d-2)\sqrt{q} + \tfrac{1}{2}d^2(d-1)\right) + d(d-1) \\
    &=d(d-1)^2 (d-2)\sqrt{q} + \frac{1}{2}d^3(d-1)^3 + d(d-1)
\end{align*}
contradicting \eqref{ineq:contradictory}. We conclude that $C^{\ast}$ is not a blocking curve, which means that $C$ is not tangent-filling.

When $C$ is irreducible but not geometrically irreducible, we know that $\# C(\Fq)\leq \frac{d^2}{4}$, so we apply the same argument (with $p$ replaced with $q$ everywhere) at the end of the proof of Theorem~\ref{thm:low-degree-p}. We conclude that $C$ is still not tangent-filling. 
 \end{proof}

 \begin{remark}\label{rem:kaji}
     Kaji~\cite{Kaj89} proved that the Gauss map of a smooth plane curve over $\overline{\mathbb{F}_q}$ must be purely inseparable. Consequently, a smooth plane curve must have finitely many bitagents. Moreoever, only smooth strange curves are conics in characteristic $2$. These observations together tell us that Theorem~\ref{thm:low-degree-q} holds for smooth curves even when the hypothesis $p\geq d$ is removed as long as $d\geq 3$. 
 \end{remark}

\section{Explicit examples of tangent-filling curves} \label{sect:family-curves}

We start with an example of a plane curve of degree $\sqrt{q}+1$ which is tangent-filling over $\Fq$ when $q$ is a square.

\begin{example}\label{example:hermitian} Let $q$ be a prime power such that $q$ is a square. The curve $\mathcal{H}_q$ defined by
$$
x^{\sqrt{q}+1} + y^{\sqrt{q}+1} + z^{\sqrt{q}+1} = 0
$$
is tangent-filling over $\mathbb{F}_q$. The curve $\mathcal{H}_q$ is known as the \emph{Hermitian curve} in the literature. It can be checked that $\mathcal{H}_q$ has exactly $(\sqrt{q})^3+1$ distinct $\Fq$-points. Moreover, the set $\mathcal{H}_q(\Fq)$ forms a \emph{unital} in the sense of combinatorial geometry, meaning that the points can be arranged into subsets of size $\sqrt{q}+1$ so that any two points of $\mathcal{H}_q(\Fq)$ lie in a unique subset. In particular, it can be shown that every $\Fq$-line meets $\mathcal{H}_q(\Fq)$ in either $1$ or $\sqrt{q}+1$ points \cite{BE08}*{Theorem 2.2}. As a result, $\mathcal{H}_q$ is a blocking curve over $\Fq$. 

To show that $\mathcal{H}_q$ is a tangent-filling curve, we let $P_0=[a:b:c]$ to be a point in $\mathbb{P}^2(\Fq)$. We are searching for a point $Q=[x_0:y_0:z_0]\in \mathcal{H}_q(\Fq)$ such that $T_Q(C)$ contains $P_0$. This is equivalent to finding $[x_0:y_0:z_0]\in \mathcal{H}_q(\Fq)$ such that
\begin{equation}\label{eq:Hermitian-blocking-1}
    x_0^{
    \sqrt{q}} a + y_0^{\sqrt{q}} b + z_0^{\sqrt{q}} c = 0.
\end{equation}
Note that the map $[x:y:z]\mapsto [x^{\sqrt{q}}:y^{\sqrt{q}}:z^{\sqrt{q}}]$ is a bijection on the set $\mathbb{P}^2(\Fq)$, and therefore also on $\mathcal{H}_q(\Fq)$ because $\mathcal{H}_q(\Fq)$ is defined over $\Fq$. Thus, there exists $[x_1:y_1:z_1]\in \mathcal{H}_q(\Fq)$ with the property that $$[x_0:y_0:z_0]=\left[x_1^{\sqrt{q}}:y_1^{\sqrt{q}}:z_1^{\sqrt{q}}\right].$$ 
In other words, it suffices to find $[x_1:y_1:z_1]\in \mathcal{H}_q(\Fq)$ such that
\begin{equation}\label{eq:Hermitian-blocking-2}
    x_1^{q} a + y_1^{q} b + z_1^{q} c = 0.
\end{equation}
Since $x_1, y_1, z_1$ are elements of $\Fq$, we see that \eqref{eq:Hermitian-blocking-2} is equivalent to
\begin{equation}\label{eq:Hermitian-blocking-3}
x_1 a + y_1 b + z_1 c = 0.
\end{equation}
Let $L$ be the $\Fq$-line defined by $ax+by+cz=0$. Since $\mathcal{H}_q(\Fq)$ is a blocking set, the equation~\eqref{eq:Hermitian-blocking-3} is satisfied for some $Q=[x_1:y_1:z_1]\in\mathcal{H}_q(\Fq)$, as claimed. This argument also shows that the dual of the Hermitian curve is isomorphic to itself. \end{example}

For the remainder of the paper, we will focus on the curve $C$ defined by the equation
\begin{equation}\label{eq:special-curve}
    x^{q-1} + y^{q-1} + z^{q-1} - 3(x+y+z)^{q-1} = 0.
\end{equation}
Unless otherwise stated, we will assume that $p=\operatorname{char}(\Fq)>3$.

We will study the curve $C$ by first finding the singular points, and then checking that $C$ is irreducible. Finally, we will prove that $C$ is tangent-filling, establishing Theorem~\ref{thm:special-curve}.

\subsection{Rational points of the curve.} We start by finding all the $\Fq$-points on $C$. 

\begin{lemma}\label{lem:rat-points}
The set $C(\Fq)$ is equal to the set of all points $[x:y:z]\in\mathbb{P}^2(\Fq)$ such that
$$
xyz(x+y+z)\neq 0.
$$
\end{lemma}

\begin{proof}
    Since $x^{q-1} = 1$ holds for every $x\in\Fq^{\ast}$, the conclusion is clear from \eqref{eq:special-curve}.
\end{proof}

\begin{corollary}
\label{cor:non-blocking}
The curve $C$ is not blocking.
\end{corollary}

\begin{proof} 
Consider the $\Fq$-line $L=\{z=0\}$. Then $C\cap L$ has no $\Fq$-points due to the condition in Lemma~\ref{lem:rat-points}. Thus, $C(\Fq)$ is not a blocking set. 
\end{proof}

\subsection{Singular points of the curve}

Our goal is to determine the singular points of the curve $C$ over $\overline{\mathbb{F}_q}$.

\begin{proposition}\label{prop:singular-points}
The curve $C$ has only one singular point, namely $[1:1:1]$.
\end{proposition}

\begin{proof}
By looking at the partial derivatives of the defining polynomial in \eqref{eq:special-curve}, any singular point $[x_0:y_0:z_0]$ of $C$ must satisfy,
\begin{equation}\label{eq:jacobian-condition}
    x_0^{q-2}=y_0^{q-2}=z_0^{q-2}=3(x_0+y_0+z_0)^{q-2}.
\end{equation}
In particular, any singular point $[x_0:y_0:z_0]\in C(\overline{\Fq})$ satisfies:
\begin{equation}
\label{eq:rat-points}
x_0y_0z_0(x_0+y_0+z_0)\ne 0.
\end{equation}
So, without loss of generality, we may assume that $z_0=1$. Thus, a potential singular point takes the form $[x_0:y_0:1]$ and satisfies $x_0y_0\neq 0$ by equation~\eqref{eq:rat-points}. Applying \eqref{eq:jacobian-condition}, we get
\begin{equation}\label{eq:sing-points-1}
x_0^{q-2} = y_0^{q-2} = 3(x_0+y_0+1)^{q-2} = 1.
\end{equation}
We begin by computing the expression $(x_0+y_0+1)^{q-2}$,
\begin{equation}\label{eq:sing-points-2}
(x_0+y_0+1)^{q-2} = \frac{(x_0+y_0+1)^{q}}{(x_0+y_0+1)^2} = \frac{1+x_0^q + y_0^q}{(1+x_0+y_0)^2}.
\end{equation}
The two equations \eqref{eq:sing-points-1} and \eqref{eq:sing-points-2} together give,
\begin{equation}\label{eq:sing-points-3}
\frac{3+3x_0^2+3y_0^2}{(1+x_0+y_0)^2} = 1.
\end{equation}
We can rearrange \eqref{eq:sing-points-3} into
$$
x_0^2+y_0^2 - x_0 y_0 - x_0 - y_0 + 1 = 0
$$
which can be expressed as a degree $2$ equation in $y_0$:
$$
y_0^2 - y_0(x_0+1) + x_0^2-x_0+1 = 0.
$$
Solving for $y_0$, we obtain
\begin{equation}\label{eq:sing-points-4}
y_0 = \frac{x_0+1 + (x_0-1)\gamma}{2}
\end{equation}
where $\gamma$ satisfies $\gamma^2 = -3$. We compute $y_0^q$ using \eqref{eq:sing-points-4}:
\begin{equation}\label{eq:sing-points-5}
y_0^q = \frac{x_0^q+1 + (x_0^q-1)\gamma^q}{2}.
\end{equation}
We also compute $y_0^2$ using \eqref{eq:sing-points-4}:
\begin{equation*}
y_0^2 = \frac{(x_0^2+2x_0+1)+2(x_0+1)(x_0-1)\gamma +  (x_0^2-2x_0+1)\cdot (-3)}{4}
\end{equation*}
which simplifies to:
\begin{equation}\label{eq:sing-points-6}
y_0^2 = \frac{-x_0^2 + 4x_0 - 1 + (x_0^2-1)\gamma}{2}.
\end{equation}
Since $y_0^{q-2}=1$ by \eqref{eq:sing-points-1}, we know that $y_0^q = y_0^{2}$. Equating \eqref{eq:sing-points-5} and \eqref{eq:sing-points-6},
\begin{equation}\label{eq:sing-points-7}
 \frac{-x_0^2 + 4x_0 - 1 + (x_0^2-1)\gamma}{2} = \frac{x_0^q+1 + (x_0^q-1)\gamma^q}{2}.
\end{equation}
We proceed by analyzing two cases, depending on whether $\gamma\in\Fq$ or $\gamma\notin\Fq$.

\textbf{Case 1.} $\gamma\in\Fq$.

In this case, we have $\gamma^q = \gamma$. Using $x_0^{q}=x_0^{2}$ which is implied by \eqref{eq:sing-points-1}, the equation~\eqref{eq:sing-points-7} yields,
\begin{equation*}
 \frac{-x_0^2 + 4x_0 - 1 + (x_0^2-1)\gamma}{2} = \frac{x_0^2+1 + (x_0^2-1)\gamma}{2}.
\end{equation*}
which simplifies to $(x_0-1)^2=0$, and so $x_0=1$. Using \eqref{eq:sing-points-4}, we obtain $y_0=1$ as well. This results in the singular point $[1:1:1]$ of the curve $C$.

\textbf{Case 2.} $\gamma\notin\Fq$.

In this case, $\gamma\in\mathbb{F}_{q^2}\setminus \mathbb{F}_q$ because  $\gamma^2=-3$. Since $\gamma^{q}$ is the Galois conjugate of $\gamma$, we have $\gamma^{q}=-\gamma$. Thus, \eqref{eq:sing-points-7} yields,
\begin{equation*}
\frac{-x_0^2 + 4x_0 - 1 + (x_0^2-1)\gamma}{2} = \frac{x_0^q+1 - (x_0^q-1)\gamma}{2}.
\end{equation*}
This simplifies (due to $x_0^q=x_0^2$) to,
$$
(x_0-1)^2 = (x_0^2-1)\gamma.
$$
We can eliminate the case $x_0=1$ because that will only bring us back to the singular point $[1:1:1]$ already analyzed in the previous case. After dividing both sides of the preceding equation by $x_0-1$, and solving for $x_0$, we get
\begin{equation}\label{eq:formula-x_0-1}
x_0 = \frac{1+\gamma}{1-\gamma}.
\end{equation}
Using the relation $\gamma^2=-3$, the formula \eqref{eq:formula-x_0-1} simplifies to,
\begin{equation}\label{eq:formula-x_0-2}
x_0 = \frac{\gamma-1}{2}.
\end{equation}
Applying \eqref{eq:sing-points-4}, we obtain
\begin{equation}\label{eq:formula-y_0}
y_0 = \frac{-\gamma-1}{2}.
\end{equation}
Since $\gamma^2=-3$, we have two solutions (once $\gamma$ is chosen, $-\gamma$ is also a solution). Thus, \eqref{eq:formula-x_0-2} and \eqref{eq:formula-y_0} allow us to conclude that there are two \emph{potential} singular points on the curve $C$:
\begin{equation*}
 \left[\frac{\gamma-1}{2} : \frac{-\gamma-1}{2}: 1 \right] \quad \text{and} \quad  \left[\frac{-\gamma-1}{2} : \frac{\gamma-1}{2}: 1 \right]
\end{equation*}
However, both of these points above satisfy $x_0+y_0+1=0$. By equation~\eqref{eq:rat-points}, none of these two points is singular on the curve $C$. 

We conclude that Case 2 does not occur after all, and the point $[1:1:1]$ is the unique singular point of $C$. \end{proof}

\subsection{Irreducibility of the curve} We begin with a general irreducibility criterion for a plane curve of degree at least $3$ with a unique singular point.

\begin{lemma}\label{lem:general-irreducibility-criterion}
Suppose that $D=\{F=0\}$ is a plane curve defined over a field $K$ with $\deg(F)\geq 3$ and a unique singular point $P_0\in D(\overline{K})$. After dehomogenizing $f(x,y)\colonequals F(x, y, 1)$ and applying translation, we may assume that $(0, 0)$ is the singular point of the affine curve $\{f=0\}$. Assume that the quadratic term $A_2(x, y)$ in the expansion of $f$ around $(0, 0)$ cannot be written as $L(x,y)^2$ for some $L(x,y)\in \overline{K}[x,y]$ (in other words, the equation $A_2(x,y)=0$ has precisely two solutions in $\mathbb{P}^1(\overline{K})$). Then the plane curve $D$ is irreducible over $\overline{K}$.
\end{lemma}

\begin{proof}
Since $(0, 0)$ is a singular point of $\{f=0\}$, we can then express
\begin{equation*}
f(x,y) = A_2(x,y) + A_3(x,y) + \ldots
\end{equation*}
where $A_i(x,y)$ is a homogeneous polynomial of degree $i$ in $x$ and $y$. By hypothesis, $A_2(x,y)$ splits over $\overline{K}$ as a product $L_1(x,y)\cdot L_2(x,y)$ of two distinct nonzero linear forms. If $f(x, y)=g(x, y)\cdot h(x, y)$ where $g(0,0)=h(0,0)=0$, then we claim that the component curves $\{g=0\}$ and $\{h=0\}$ meet at the point $(0,0)$ with multiplicity $1$. Indeed, the expansions of $g(x,y)$ and $h(x,y)$ around the origin $(0,0)$ must necessarily take the form (after multiplication by a suitable nonzero constant):
\begin{align*}
g(x,y) = L_1(x,y) + B_2(x,y) + B_3(x,y) + \ldots 
\end{align*}
and
\begin{align*}
h(x,y) = L_2(x,y) + C_2(x,y) + C_3(x,y) + \ldots
\end{align*}
respectively, where $B_i(x,y)$ and $C_i(x,y)$ are homogeneous polynomials of degree $i$ in
$x$ and $y$. Since $L_1(x,y)$ and $L_2(x,y)$ are distinct linear forms which generate the maximal ideal of $\overline{K}[x,y]$ at $(0,0)$, then the two curves $\{g=0\}$ and $\{h=0\}$ meet with multiplicity $1$ at $(0,0)$.

We show that the plane curve $D=\{F=0\}$ is irreducible over $\overline{K}$. Assume, to the contrary, that $F=G\cdot H$ for some homogeneous polynomials $G$ and $H$ with positive degrees $d_1$ and $d_2$, respectively. Let $g(x,y)\colonequals G(x, y, 1)$ and $h(x, y)\colonequals H(x, y, 1)$. After applying B\'ezout's theorem, $d_1 d_2$ intersection points (counted with multiplicity) of $\{G=0\}$ and $\{H=0\}$ must be singular points of $D$. Since $D$ has a unique singular point, namely $(0, 0)$ in the affine chart $z=1$, the local intersection multiplicity at the origin must be at least $d_1 d_2 \geq 2$.  This contradicts the fact that $\{g=0\}$ and $\{h=0\}$ meet with multiplicity exactly $1$ at $(0, 0)$.
\end{proof}

\begin{proposition}\label{prop:irreducibility-of-special-curve}
The curve $C$ defined by \eqref{eq:special-curve} is geometrically irreducible.
\end{proposition}

\begin{proof}
By Proposition~\ref{prop:singular-points}, the curve $C$ has the unique singular point $[1:1:1]$. Expanding the equation $x^{q-1}+y^{q-1}+1-3(x+y+1)^{q-1}=0$ around the point $(1, 1)$, we are led to analyze:
\begin{equation*}
(1 + (x-1))^{q-1} + (1 + (y-1))^{q-1} + 1 - 3(3 + (x-1) + (y-1))^{q-1} 
\end{equation*}
After expanding, the first nonzero homogeneous form in $(x-1)$
and $(y-1)$ has degree $2$, and is given by:
\begin{equation*}
2\cdot 3^{q-2}\cdot \left[(x-1)^2 - (x-1)(y-1) + (y-1)^2 \right].
\end{equation*}
Since the discriminant of the quadratic $s^2-st+t^2$ is $-3\neq 0$ in $\Fq$, the hypothesis of Lemma~\ref{lem:general-irreducibility-criterion} is satisfied. Thus, $C$ is irreducible over $\overline{\Fq}$.
\end{proof}

\subsection{Tangent-filling property} In this final subsection, we give the proof that the curve $C$ defined by \eqref{eq:special-curve} is tangent-filling over $\Fq$. 

\begin{proof}[Proof of Theorem~\ref{thm:special-curve}] Let $P=[a_0:b_0:c_0]$ be an arbitrary point in $\bP^2(\Fq)$. We want to find a smooth $\Fq$-point $Q=[x_0:y_0:z_0]$ of $C$ such that $P$ is contained in the tangent line $T_Q C$. From Lemma~\ref{lem:rat-points}, we know that an $\Fq$-point $[x_0:y_0:z_0]$ is a point on the curve $C$ if and only if 
\begin{equation}\label{eq:rat-point-condition}
x_0y_0z_0(x_0+y_0+z_0)\neq 0
\end{equation}
Note that $P$ is contained in the tangent line $T_Q C$ if and only if
\begin{equation}\label{eq:tangent-condition}
a_0\cdot \left(3 s_0^{q-2}-x_0^{q-2}\right)+ b_0\cdot \left(3s_0^{q-2}-y_0^{q-2}\right)+ c_0\cdot \left(3s_0^{q-2} - z_0^{q-2}\right)=0
\end{equation}
where $s_0=x_0+y_0+z_0$. Using the fact that $s^{q-1}=1$ for each $s\in\Fq^{\ast}$, we rewrite \eqref{eq:tangent-condition} as
\begin{equation}\label{eq:tangent-condition-fraction}
\frac{3(a_0+b_0+c_0)}{x_0+y_0+z_0}=\frac{a_0}{x_0}+\frac{b_0}{y_0} +\frac{c_0}{z_0}.
\end{equation}
Note that all the denominators in \eqref{eq:tangent-condition-fraction} are nonzero because Lemma~\ref{lem:rat-points} guarantees that $xyz(x+y+z)\neq 0$ for any $\Fq$-point $[x:y:z]$ of the curve $C$. 

\textbf{Case 1.} Suppose $a_0b_0c_0(a_0+b_0+c_0)\ne 0$ and $[a_0:b_0:c_0]\ne [1:1:1]$. 

In this case,  the point $P=[a_0:b_0:c_0]$ is already smooth in $C(\Fq)$ by Lemma~\ref{lem:rat-points} and Proposition~\ref{prop:singular-points}. Hence, we may take $Q=P$ because $P$ always belongs to $T_P C$.

\textbf{Case 2.} Suppose $a_0+b_0+c_0=0$.

In this case, \eqref{eq:tangent-condition-fraction} yields 
\begin{equation}\label{eq:case2}
\frac{a_0}{x_0}+\frac{b_0}{y_0}+\frac{c_0}{z_0}=0.
\end{equation}
We search for a solution $[x_0:y_0:z_0]\ne [1:1:1]$ satisfying~\eqref{eq:rat-point-condition}. 

\textbf{Subcase 2.1.} $a_0+b_0+c_0=0$ and $a_0b_0c_0\ne 0$.

Since $\operatorname{char}(\Fq)>3$, we cannot have $a_0=b_0=c_0$. We may assume, without loss of generality, that $b_0\ne c_0$. Let $z_0=1$ and $y_0=-1$, and solve for $x_0$ according to the equation~\eqref{eq:case2}:
$$
x_0=\frac{a_0}{b_0-c_0}\in \Fq^{\ast}
$$
Clearly, $[x_0:y_0:z_0]\ne [1:1:1]$ and ~\eqref{eq:rat-point-condition} is satisfied. 

\textbf{Subcase 2.2.} $a_0+b_0+c_0=0$ and $a_0b_0c_0=0$.

By symmetry, we may assume that $a_0=0$; since $a_0+b_0+c_0=0$, then we have $[a_0:b_0:c_0]=[0:1:-1]$ and so, equation~\eqref{eq:case2} yields $y_0=z_0$. The point $[x_0:y_0:z_0]=[2:1:1]$ satisfies both \eqref{eq:rat-point-condition}~and~\eqref{eq:case2}. This concludes our proof that all points $[a_0:b_0:c_0]$ for which $a_0+b_0+c_0=0$ belong to a tangent line at a smooth $\Fq$-point of $C$.

\textbf{Case 3.} $a_0+b_0+c_0\neq 0$ and $a_0b_0c_0=0$. 

Since we seek points $[x_0:y_0:z_0]$ with $x_0+y_0+z_0\neq 0$, we can scale $[a_0:b_0:c_0]$ and $[x_0:y_0:z_0]$ so that
$$
a_0 + b_0 + c_0 = 3 \quad\quad \text{ and } \quad\quad x_0+y_0+z_0=9
$$
The equation~\eqref{eq:tangent-condition-fraction} now reads,
\begin{equation}\label{eq:case3}
1=\frac{a_0}{x_0}+\frac{b_0}{y_0}+\frac{3-a_0-b_0}{9-x_0-y_0};
\end{equation}
Since $a_0b_0c_0=0$, we may assume by symmetry that $a_0=0$. As a result, \eqref{eq:case3} reads
\begin{equation}
\label{eq:case3_a0=0}
1=\frac{b_0}{y_0} + \frac{3-b_0}{z_0}.
\end{equation}
If $b_0\notin \{0,-3,3\}$, then we let $z_0=6$, $y_0=6b_0/(3+b_0)$ and $x_0=(9-3b_0)/(3+b_0)$. Note that $[x_0:y_0:z_0]\neq [1:1:1]$ and satisfies both \eqref{eq:case3_a0=0} and ~\eqref{eq:rat-point-condition}.

If $b_0=0$, then we simply choose $[x_0:y_0:z_0]=[2:4:3]\neq [1:1:1]$ which satisfies both \eqref{eq:case3_a0=0} and ~\eqref{eq:rat-point-condition}.

If $b_0=-3$, then we get the solution $[x_0:y_0:z_0]=[-1:6:4]\ne [1:1:1]$ which satisfies both \eqref{eq:case3_a0=0} and ~\eqref{eq:rat-point-condition}.

If $b_0=3$, then we get the solution $[x_0:y_0:z_0]=[2:3:4]\ne [1:1:1]$ which satisfies both \eqref{eq:case3_a0=0} and equation~\eqref{eq:rat-point-condition}.

\textbf{Case 4.} $[a_0:b_0:c_0]=[1:1:1]$.

We can assume $a_0=b_0=c_0=1$, and also $x_0+y_0+z_0=9$ after scaling $[x_0:y_0:z_0]$. Then equation~\eqref{eq:tangent-condition-fraction} yields,
\begin{equation}\label{eq:case4}
1=\frac{1}{x_0}+\frac{1}{y_0}+\frac{1}{9-x_0-y_0}.
\end{equation}
Our goal is to find a solution $(3,3)\neq (x_0,y_0)\in\Fq^*\times \Fq^*$ to \eqref{eq:case4}.

After multiplying \eqref{eq:case4} by $x_0y_0(9-x_0-y_0)$, we obtain
$$
9x_0y_0-x_0^2y_0-x_0y_0^2=9y_0-x_0y_0-y_0^2+9x_0-x_0^2-x_0y_0+x_0y_0,
$$
which we rearrange as follows:
$$
y_0^2(x_0-1)+y_0(x_0-1)(x_0-9)-x_0(x_0-9)=0.
$$
Our goal is to show that the number of $\Fq$-points on the affine curve $Y$ given by the equation:
\begin{equation}\label{eq:case4-affine-curve}
y^2(x-1)+y(x-1)(x-9)-x(x-9)=0
\end{equation}
is strictly more than the number of points which we want to avoid from the set:
\begin{equation}\label{eq:avoid-points}
\{(0,9), (0,0), (9,0), (3,3)\}.
\end{equation}
Indeed, besides the point $(3,3)$,  the points $(x_0,y_0)$ on the curve~\eqref{eq:case4-affine-curve} which we have to avoid are the ones satisfying the equation:  
$$x_0y_0\cdot (x_0+y_0-9)=0.$$ 
We note that there are only three such points on the curve~\eqref{eq:case4-affine-curve}: $(0,0)$, $(0,9)$ and $(9,0)$; this follows easily from the equation~\eqref{eq:case4-affine-curve} after substituting either $x=0$, or $y=0$, or $x=9-y$.

Now, for each $\Fq$-point $(x_0,w_0)\ne (1,0)$ on the affine conic $\tilde{Y}$ given by the equation
$$w^2=(x-1)(x-9),$$
we have the $\Fq$-point $(x_0,y_0)$ on $Y$ given by
\begin{equation}\label{eq:affine-curve-solution}
(x_0,y_0)\colonequals \left(x_0,\frac{-(x_0-1)(x_0-9)+(x_0-3)w_0}{2(x_0-1)}\right).
\end{equation}
Since there are $q-2$ points $(x_0,w_0)\ne (1,0)$ on $\tilde{Y}(\Fq)$ (because we have $q+1$ points on its projective closure in $\bP^2$ and only two such points are on the line at infinity), we obtain $(q-2)$ $\Fq$-points on $Y$. Now, if $(x_1,w_1)\ne (x_0,w_0)$ are distinct points on $\tilde{Y}(\Fq)\setminus\{(1,0)\}$, then we get the corresponding points on $Y(\Fq)$ are also distinct \emph{unless} $x_0=x_1=3$ as can be seen from ~\eqref{eq:affine-curve-solution}. There are at most $2$ points on $\tilde{Y}(\Fq)$ having $x$-coordinate equal to $3$ (which in fact happens when $q=7$, in which case $(3,\pm 3)\in \tilde{Y}(\mathbb{F}_7)$). Thus, we are guaranteed to have at least $(q-3)$ distinct points in $Y(\Fq)$. Hence, as long as $q>7$, we are guaranteed to avoid the points listed in \eqref{eq:avoid-points}.

Therefore, the curve $C$ is tangent-filling under the hypothesis $q>7$ and $\operatorname{char}(\Fq)>3$. \end{proof}

\begin{remark} 
\label{rem:q_7}
The result in Theorem~\ref{thm:special-curve} is sharp in a sense that when $q=7$, the curve $x^{q-1}+y^{q-1}+z^{q-1}-3(x+y+z)^{q-1}=0$ is \emph{not} tangent-filling. Indeed, one can check that for the point $P=[1:1:1]$, there is no smooth $\mathbb{F}_7$-point $Q$ on this curve $C$ such that  $P\in T_Q C$. 
\end{remark}


\begin{bibdiv}
\begin{biblist}





\bib{AG23}{article}{
    AUTHOR = {Asgarli, Shamil},
    AUTHOR = {Ghioca, Dragos},
     TITLE = {Smoothness in pencils of hypersurfaces
              over finite fields},
   JOURNAL = {Bull. Aust. Math. Soc.},
    VOLUME = {107},
      YEAR = {2023},
    NUMBER = {1},
     PAGES = {85--94},
      ISSN = {0004-9727},
}

\bib{AGY22a}{article}{
Author ={Asgarli, Shamil},
Author = {Ghioca, Dragos},
Author = {Yip, Chi Hoi},
title={Blocking sets arising from plane curves over finite fields},
journal = {arXiv e-prints},
 eprint = {https://arxiv.org/abs/2208.13299},
year={2022},
}

\bib{AGY22b}{article}{
Author ={Asgarli, Shamil},
Author = {Ghioca, Dragos},
Author = {Yip, Chi Hoi},
title={Most plane curves over finite fields are not blocking},
journal = {arXiv e-prints},
 eprint = {https://arxiv.org/abs/2211.08523},
year={2022},
}

\bib{AGY23a}{article}{
Author ={Asgarli, Shamil},
Author = {Ghioca, Dragos},
Author = {Yip, Chi Hoi},
title={Proportion of blocking curves in a pencil},
journal = {arXiv e-prints},
 eprint = {https://arxiv.org/abs/2301.06019},
year={2023}
}

\bib{AGY23b}{article}{
Author ={Asgarli, Shamil},
Author = {Ghioca, Dragos},
Author = {Yip, Chi Hoi},
title={Existence of pencils with nonblocking hypersurfaces},
journal = {arXiv e-prints},
 eprint = {https://arxiv.org/abs/2301.09215},
year={2023}
}

\bib{AP96}{incollection}{
    AUTHOR = {Aubry, Yves},
    AUTHOR = {Perret, Marc},
     TITLE = {A {W}eil theorem for singular curves},
 BOOKTITLE = {Arithmetic, geometry and coding theory ({L}uminy, 1993)},
     PAGES = {1--7},
 PUBLISHER = {de Gruyter, Berlin},
      YEAR = {1996},
}

\bib{BH91}{article}{
    AUTHOR = {Bayer, Valmecir},
    AUTHOR = {Hefez, Abramo},
     TITLE = {Strange curves},
   JOURNAL = {Comm. Algebra},
    VOLUME = {19},
      YEAR = {1991},
    NUMBER = {11},
     PAGES = {3041--3059},
      ISSN = {0092-7872},
}


\bib{Blo94}{article}{
     AUTHOR = {Blokhuis, Aart},
     TITLE = {On the size of a blocking set in {${\rm PG}(2,p)$}},
   JOURNAL = {Combinatorica},
    VOLUME = {14},
      YEAR = {1994},
    NUMBER = {1},
     PAGES = {111--114},
      ISSN = {0209-9683},
}

\bib{BE08}{book}{
    AUTHOR = {Barwick, Susan},
    AUTHOR = {Ebert, Gary},
     TITLE = {Unitals in projective planes},
    SERIES = {Springer Monographs in Mathematics},
 PUBLISHER = {Springer, New York},
      YEAR = {2008},
     PAGES = {xii+193},
}

\bib{BLM23}{article}{
    AUTHOR = {Beelen, Peter},
    AUTHOR = {Landi, Leonardo},
    AUTHOR = {Montanucci, Maria},
     TITLE = {Classification of all {G}alois subcovers of the {S}kabelund
              maximal curves},
   JOURNAL = {J. Number Theory},
    VOLUME = {242},
      YEAR = {2023},
     PAGES = {46--72},
      ISSN = {0022-314X},
}

\bib{BM18}{article}{
    AUTHOR = {Beelen, Peter},
    AUTHOR = {Montanucci, Maria},
     TITLE = {A new family of maximal curves},
   JOURNAL = {J. Lond. Math. Soc. (2)},
    VOLUME = {98},
      YEAR = {2018},
    NUMBER = {3},
     PAGES = {573--592},
      ISSN = {0024-6107},
}
		


\bib{CHKT00}{article}{
    AUTHOR = {Cossidente, A.},
    AUTHOR = {Hirschfeld, J. W. P.},
    AUTHOR = {Korchm\'{a}ros, G.},
    AUTHOR = {Torres, F.},
     TITLE = {On plane maximal curves},
   JOURNAL = {Compositio Math.},
    VOLUME = {121},
      YEAR = {2000},
    NUMBER = {2},
     PAGES = {163--181},
      ISSN = {0010-437X},
}

\bib{CM06}{article}{
    AUTHOR = {Cafure, Antonio},
    AUTHOR = {Matera, Guillermo},
     TITLE = {Improved explicit estimates on the number of solutions of
              equations over a finite field},
   JOURNAL = {Finite Fields Appl.},
    VOLUME = {12},
      YEAR = {2006},
    NUMBER = {2},
     PAGES = {155--185},
      ISSN = {1071-5797},
}

\bib{Dur18}{article}{
    AUTHOR = {Duran Cunha, Gregory},
     TITLE = {Curves containing all points of a finite projective {G}alois
              plane},
   JOURNAL = {J. Pure Appl. Algebra},
    VOLUME = {222},
      YEAR = {2018},
    NUMBER = {10},
     PAGES = {2964--2974},
      ISSN = {0022-4049},
}

\bib{GGS10}{article}{
    AUTHOR = {Garcia, Arnaldo},
    AUTHOR = {G\"{u}neri, Cem},
    AUTHOR = {Stichtenoth, Henning},
     TITLE = {A generalization of the {G}iulietti-{K}orchm\'{a}ros maximal
              curve},
   JOURNAL = {Adv. Geom.},
    VOLUME = {10},
      YEAR = {2010},
    NUMBER = {3},
     PAGES = {427--434},
      ISSN = {1615-715X},
}

\bib{GK09}{article}{
    AUTHOR = {Giulietti, Massimo}, 
    AUTHOR = {Korchm\'{a}ros, G\'{a}bor},
     TITLE = {A new family of maximal curves over a finite field},
   JOURNAL = {Math. Ann.},
    VOLUME = {343},
      YEAR = {2009},
    NUMBER = {1},
     PAGES = {229--245},
      ISSN = {0025-5831},
}

\bib{Har92}{book}{
    AUTHOR = {Harris, Joe},
     TITLE = {Algebraic geometry},
    SERIES = {Graduate Texts in Mathematics},
    VOLUME = {133},
      NOTE = {A first course},
 PUBLISHER = {Springer-Verlag, New York},
      YEAR = {1992},
     PAGES = {xx+328},
      ISBN = {0-387-97716-3},
}

\bib{Hir79}{article}{
    AUTHOR = {Hirschfeld, J. W. P.},
     TITLE = {Projective geometries over finite fields},
    SERIES = {Oxford Mathematical Monographs},
 PUBLISHER = {The Clarendon Press, Oxford University Press, New York},
      YEAR = {1979},
     PAGES = {xii+474},
      ISBN = {0-19-853526-0},
}

\bib{Hom20}{article}{
    AUTHOR = {Homma, Masaaki},
     TITLE = {Fragments of plane filling curves of degree {$q+2$} over the
              finite field of {$q$} elements, and of affine-plane filling
              curves of degree {$q+1$}},
   JOURNAL = {Linear Algebra Appl.},
    VOLUME = {589},
      YEAR = {2020},
     PAGES = {9--27},
      ISSN = {0024-3795},
}

\bib{HK13}{article}{
    AUTHOR = {Homma, Masaaki},
    AUTHOR = {Kim, Seon Jeong},
     TITLE = {Nonsingular plane filling curves of minimum degree over a
              finite field and their automorphism groups: supplements to a
              work of {T}allini},
   JOURNAL = {Linear Algebra Appl.},
    VOLUME = {438},
      YEAR = {2013},
    NUMBER = {3},
     PAGES = {969--985},
      ISSN = {0024-3795},
}

\bib{HKT08}{book}{
    AUTHOR = {Hirschfeld, J. W. P.},
    AUTHOR = {Korchm\'{a}ros, G.},
    AUTHOR = {Torres, F.},
     TITLE = {Algebraic curves over a finite field},
    SERIES = {Princeton Series in Applied Mathematics},
 PUBLISHER = {Princeton University Press, Princeton, NJ},
      YEAR = {2008},
     PAGES = {xx+696},
      ISBN = {978-0-691-09679-7},
}

\bib{Kaj89}{article}{
    AUTHOR = {Kaji, Hajime},
     TITLE = {On the {G}auss maps of space curves in characteristic {$p$}},
   JOURNAL = {Compositio Math.},
    VOLUME = {70},
      YEAR = {1989},
    NUMBER = {2},
     PAGES = {177--197},
      ISSN = {0010-437X},
       URL = {http://www.numdam.org/item?id=CM_1989__70_2_177_0},
}

\bib{Par86}{article}{
    AUTHOR = {Pardini, Rita},
     TITLE = {Some remarks on plane curves over fields of finite
              characteristic},
   JOURNAL = {Compositio Math.},
    VOLUME = {60},
      YEAR = {1986},
    NUMBER = {1},
     PAGES = {3--17},
      ISSN = {0010-437X},
}

\bib{Pie94}{article}{
    AUTHOR = {Piene, Ragni},
     TITLE = {Projective algebraic geometry in positive characteristic},
 BOOKTITLE = {Analysis, algebra, and computers in mathematical research
              ({L}ule\aa , 1992)},
    SERIES = {Lecture Notes in Pure and Appl. Math.},
    VOLUME = {156},
     PAGES = {263--273},
 PUBLISHER = {Dekker, New York},
      YEAR = {1994},
}

\bib{Wal56}{article}{
    AUTHOR = {Wallace, Andrew H.},
     TITLE = {Tangency and duality over arbitrary fields},
   JOURNAL = {Proc. London Math. Soc. (3)},
    VOLUME = {6},
      YEAR = {1956},
     PAGES = {321--342},
      ISSN = {0024-6115},
}




\end{biblist}
\end{bibdiv}
\end{document}